\documentclass[12pt]{article}
\usepackage{amsthm,amsmath,amssymb}
\newtheorem{theorem}{Theorem}
\newtheorem{lemma}[theorem]{Lemma}
\newtheorem{cor}[theorem]{Corollary}
\newtheorem{prop}[theorem]{Proposition}

\theoremstyle{definition}
\newtheorem{definition}[theorem]{Definition}

\newtheorem{conjecture}[theorem]{Conjecture}

\newtheorem{remark}[theorem]{Remark}
\theoremstyle{remark}

\DeclareMathOperator{\lcc}{lcc}
\DeclareMathOperator{\scp}{scp}
\DeclareMathOperator{\cp}{cp}

\title{\bf Clique Coverings and Claw-free Graphs}
\author{Csilla Bujt\'{a}s\thanks{ Department of Computer Science and Systems Technology,
University of Pannonia, Veszpr\'em, Hungary;
 and
  MTA R\'enyi Institute, Budapest,
 Hungary. E-mail: {\tt bujtas@dcs.uni-pannon.hu}, {\tt
tuza@dcs.uni-pannon.hu}}
 \qquad Akbar Davoodi\thanks{ Department of
Mathematical Sciences, Isfahan University of Technology,
84156-83111, Isfahan, Iran. E-mail: {\tt
akbar.davoodi@math.iut.ac.ir}}
  \qquad Ervin Gy\H{o}ri\thanks{MTA R\'enyi Institute and Central European University, Budapest,
 Hungary,  Research partially supported by the NKFIH Grant 116769, E-mail: {\tt gyori.ervin@renyi.mta.hu}}
  \qquad Zsolt Tuza$^*$\,\\[4pt]
}

\date{}

\begin{document}

\maketitle

\begin{abstract}
Let $\cal C$ be a clique covering for $E(G)$ and let $v$ be a vertex of $G$.
 The valency of vertex $v$ (with respect to $\cal C$), denoted by $val_{\cal C}(v)$, is the number of cliques in $\cal C$ containing $v$. The local clique cover number of $G$, denoted by $\lcc(G)$, is defined as the smallest integer $k$, for which there exists a clique covering for $E(G)$ such that $val_{\cal C}(v)$ is at most $k$, for every vertex $v\in V(G)$.
In this paper, among other results, we prove that if $G$ is a claw-free graph, then $\lcc(G)+\chi(G)\leq n+1$.

  \bigskip\noindent \textbf{Keywords:} edge clique covering; local clique covering; chromatic number; claw-free graph; sigma clique partition; Nordhaus-Gaddum inequality.\\
\textbf{MSC:} 05C70
\end{abstract}
\section{Introduction}
Throughout the paper, all graphs are simple and undirected. By a \textit{clique} of a graph $G$, we mean a subset of mutually adjacent vertices of $ G $ as well as its corresponding complete subgraph. The \textit{size} of a clique is the number of its vertices.
A \textit{clique covering} for $E(G)$ is defined as a family of cliques of $G$ such that every edge of $G$ lies in at least one of the cliques comprising this family.

Let $\cal C$ be a clique covering for $E(G)$ and let $v$ be a vertex of $G$. {\it Valency} of vertex $v$ (with respect to $\cal C$), denoted by $val_{\cal C}(v)$, is defined to be the number of cliques in $\cal C$ containing~$v$. A number of different variants of the clique cover number have been investigated in the literature.
The {\it local clique cover number} of $G$, denoted by $\lcc(G)$, is defined as the smallest integer $k$, for which there exists a clique covering for $G$ such that $val_{\cal C}(v)$ is at most $k$, for every vertex $v\in V(G)$.

This parameter may be interpreted as a variety of different invariants of the graph
 and the problem relates to some well-known problems such as line graphs of hypergraphs, intersection representation and Kneser representation of graphs.
For example, $\lcc(G)$ is the minimum integer $k$ such that $G$ is the line graph of a $k$-uniform hypergraph. By this interpretation, $\lcc(G)\leq 2$ if and only if $G$ is the line graph of a multigraph.

There is a characterization by a list of seven forbidden induced subgraphs and
a polynomial-time algorithm for the recognition that $G$ is the line graph of a multigraph \cite{Bermond, Lehot}.
 On the other hand, L. Lov\'{a}sz in \cite{Lovasz} proved that there is no characterization by a finite list of forbidden induced subgraphs for the graphs which are line graphs of some $3$-uniform hypergraphs. Moreover, it was proved that the decision problem whether $G$ is the line graph of a $k$-uniform hypergraph, for fixed $k\geq4$, and the problem of recognizing line graphs of $3$-uniform hypergraphs without multiple edges are NP-complete \cite{Poljak}.

 For a vertex $v\in V(G)$, its {\it (open) neighborhood} $N(v)$ is the set of all neighbors of $v$ in $G$, while its {\it closed neighborhood}
  $N[v]$ is defined as
 $N[v]:=N(v) \cup \{v\}$. Moreover, let $\overline{G}$ stand for the complement of $G$, and let $\Delta(G)$ and $\delta(G)$ be the maximum degree and the minimum degree of $G$, respectively.
The subgraph induced by a set $Y\subset V(G)$ will be denoted by $G[Y]$.
By the notations of $\alpha(G)$, $\omega(G)$, and $\chi(G)$ we mean the independence number, clique number, and chromatic number of $G$, respectively.

 In $ 1956 $ E. A. Nordhaus and J. W. Gaddum proved the following theorem for the chromatic number of a graph $ G $ and its complement, $ \overline{G} $.
\begin{theorem}{\em\cite{nord}}\label{thm:NG}
Let $ G $ be a graph on $ n $ vertices. Then
$ 2\sqrt{n}\leq\chi(G)+\chi(\overline{G})\leq n+1 $.
\end{theorem}
Later on, similar results for other graph parameters have been found which are known as Nordhaus-Gaddum type theorems. In the literature there are several hundred papers considering inequalities of this type for many other graph invariants. For a survey of Nordhaus-Gaddum type estimates see \cite{SurveyNG}.

In this paper, we consider the following two conjectures on local clique cover number.

\begin{conjecture}{\label{conj1}}
For every graph $G$ on $n$ vertices,
\begin{align}\label{eq:conj1}
\lcc(G)+\lcc(\overline{G})\leq n.
\end{align}
\end{conjecture}
This conjecture proposed by R. Javadi, Z. Maleki and B. Omoomi in $2012$. Note that Conjecture \ref{conj1} is a Nordhaus-Gaddum type inequality concerning the local clique cover number of $G$.

The second author with R. Javadi and B. Omoomi suggested the following weakening of Conjecture 2.
\begin{conjecture}{\label{conj2}}
For every graph $G$ on $n$ vertices,
\begin{align}\label{eq:conj2}
\lcc(G)+\chi(G)\leq n+1.
\end{align}
\end{conjecture}


Let $G_1$ and $G_2$ be graphs with disjoint vertex sets $V(G_1)$ and
$V(G_2)$ and edge sets $E(G_1)$ and $E(G_2)$.
The disjoint union of $G_1$ and $G_2$, denoted by $G_1\dot\cup G_2$, is the graph with vertex set $V(G_1)\cup V(G_2)$ and edge set $E(G_1)\cup E(G_2)$.

\begin{lemma}\label{lem:conj1=>conj2}
Let $\cal G$ be a family of graphs which is closed under the operation of taking disjoint union with an isolated vertex. If Conjecture \ref{conj1} is true for every  $G\in{\cal G}$, then Conjecture \ref{conj2} is also true for every  $G\in{\cal G}$.
\end{lemma}
\begin{proof}
Let $G\in{\cal G}$ and consider the disjoint union $H=G\dot\cup
\{v\}$. Observe that $\lcc(G)=\lcc(H)$. Hence, assuming that each
member of $\cal G$ satisfies Conjecture \ref{conj1}, we have
$\lcc(G)+\lcc(\overline{H}) \le |V(H)|$. Now, fix a clique covering
$\cal C$ for $\overline{H}$. Clearly, $\chi(G)\leq val_{\cal
C}(v)\leq \lcc(\overline{H})$. These two inequalities together imply
$\lcc(G)=\lcc(H)\leq |V(H)|-\lcc(\overline{H})\le |V(G)|+1-\chi(G)$.
\end{proof}

\section{Proof of some variants of the conjectures}

Let $k$ be an integer and let $G$ be a graph such that $k\leq\deg(x)\leq k+1$, for every vertex $x\in V(G)$. Then $\lcc(G)\leq k+1$ and $\lcc(\overline{G})\leq n-1-k$. Thus, inequality \eqref{eq:conj1} holds for $G$. Also, If $G$ is a triangle-free graph, then for a vertex $v$ which has the maximum degree in $G$, $N(v)$ can properly be colored by one color. Thus, $\chi(G)\leq n+1-\Delta(G)$. Since $\lcc(G)=\Delta(G)$, Conjecture \ref{conj2} is true for triangle-free graphs. In what follows we prove that not only \eqref{eq:conj2} but also \eqref{eq:conj1} holds if $\overline{G}$ is triangle-free.

\begin{theorem}\label{thm:alpha=2}
Let $G$ be a graph on $n$ vertices. If $\alpha(G)=2$, then $\lcc(G)+\lcc(\overline{G})\leq n$.
\end{theorem}
\begin{proof}
Clearly, $\lcc(\overline{G})\leq\Delta(\overline{G})=n-1-\delta(G)$. It is enough to show that $\lcc(G)\leq\delta(G)+1$. Let $v$ be a vertex of minimum degree in $G$, and let $K\subset V(G)$ be the set of vertices which are not adjacent to $v$. Since $\alpha(G)=2$, the induced subgraph on $K$, $G[K]$, is a clique in $G$. Now, for every vertex $u_i\in N(v)$, let $C_i:=(N(u_i)\cap K)\cup \{u_i\}$ and define $C_{\delta(G)+1}:=G[K]$. These cliques along with the collection of those edges which are not covered by the cliques $C_1,\ldots, C_{\delta(G)+1}$ comprise a clique covering for $G$, say $\cal C$. It can be checked easily that $val_{\cal C}(v)=\delta(G)$ and $val(x)\leq\delta(G)+1$, for every vertex $x\in V(G)-v$.
\end{proof}
It is well-known that $\frac{n}{\alpha(G)}$ and $\omega(G)$ are
lower bounds for $\chi(G)$, the chromatic number of $G$.  We show
that, if we replace $\chi(G)$ with any of these two general lower
bounds in Conjecture \ref{conj2}, then the inequality holds.

\begin{prop}\label{prop:omega}
Let $G$ be a graph with $n$ vertices. Then $\lcc(G)+\omega(G)\leq
n+1$.
\end{prop}
\begin{proof}
Assume that $K\subset V(G)$ is a clique of size $\omega$. For every
vertex $v_i\in V(G)-K$, $1\leq i\leq n-\omega$, define
$C_i:=(N(v_i)\cap K)\cup \{v_i\}$, and let $C_{n-\omega+1}:=G[K]$. Now, let
$F$ be the set of all the edges which are not covered by the cliques
$C_1,\ldots, C_{n-\omega +1}$. Clearly, the cliques
$C_i$ for $1\leq i\leq n-\omega +1$ together  with $F$ form a
clique covering $\mathcal{C}$ for $G$. If $x\in K$, then $val_{\cal C}(x)\leq
1+n-\omega(G)$, and for vertex $v_i\in V(G)-K$, $val_{\cal
C}(v_i)\leq n-\omega(G).$
\end{proof}

Before proving the other inequality $\lcc(G)+\frac{n}{\alpha(G)}\leq
n+1$, we  verify a stronger statement involving local parameters.
Let $\alpha_G(v)=\alpha(G[N(v)])$ be the maximum number of
independent vertices in the neighborhood of vertex $v$, and let
the {\it local independence number} of graph $G$ be defined as
$\alpha_L(G)=\max_{v\in V(G)}\alpha_G(v)$. Clearly, $\alpha_G(v) \le
\alpha_L(G) \le \alpha(G)$. Further, $\alpha_G(v)\ge 1$ holds if and
only if $v$ has at least one neighbor, while  $\alpha_G(v)\le 1$ is
equivalent to that the {\it closed neighborhood} $N_G[v]=N(v) \cup
\{v\}$ induces a clique.
\begin{theorem}\label{thm:alpha-local}
For every graph $G$ of order $n$, there exists a clique covering
$\cal C$ such that for each non-isolated vertex $v \in V(G)$ the
inequality $val_{\cal C}(v)+\frac{n}{\alpha_G(v)}\leq n+1$ holds.
\end{theorem}
\begin{proof}
A clique covering will be called {\it good\/}  if it satisfies the
requirement given in the theorem. Since the statement  is true for
all graphs of order $n \le 3$, we may proceed by induction on $n$.
Let $x$ and $y$ be two adjacent vertices of $G$. By the induction
hypothesis,
  there is a good clique covering, $\cal C'$, for
  $G'=G-\{x,y\}$. We introduce the  notations $N_1:=N(x)-N[y]$, $N_2:=N(y)-N[x]$, and $N_{1,2}:=N(x)\cap N(y)$.
  To obtain a good clique covering $\cal C$ of
  $G$ from $\cal C'$, we perform the following steps.
  \begin{enumerate}
  \item To handle vertices whose neighbors are completely adjacent,
   observe that every vertex $u$ from $N_1 \cup N_2 \cup N_{1,2}$ with $\alpha_G(u)=1$
  and $\deg_{G'}(u)\ge 1$ has $\alpha_{G'}(u)=1$ and hence it is
  covered by the clique $N_{G'}[u]$ in the good covering $\cal C'$. Now, for each such
  vertex $u$, $N_{G'}[u]$ is extended by $x$, by $y$ or by both $x$
  and $y$ respectively, if $u \in N_1$, $u\in N_2$ or $u\in
  N_{1,2}$.
  \item If $\alpha_G(x)=1<\alpha_G(y)$, take the clique $N_G[x]$; if
  $\alpha_G(y)=1<\alpha_G(x)$, take the clique $N_G[y]$; and if
    $\alpha_G(x)=\alpha_G(y)=1$, take the clique
  $N_G[x]=N_G[y]$ into the covering $\cal C$ (if they were not included in step $(1)$).
  \item If there still exist some uncovered edges between $x$ and
  $N_1$, we consider the set $N_1'= \{v\in N_1 \mid xv \mbox{ is
  uncovered}\}$ and partition it into some number of
   adjacent vertex pairs (inducing independent edges) and
  at most $\alpha(G(N_1'))$ isolated vertices. Then, we extend each of
  them with $x$ to a $K_3$ or $K_2$, and insert these cliques into the covering $\cal
  C$. This way, we get at most $\frac{|N_1'|-\alpha(G(N_1'))}{2}+\alpha(G(N_1'))$ new cliques.
  Then, we define  $N_2'$ and $N_{1,2}'$ analogously, and do the corresponding partitioning   procedure for $N_2'$ and
  $N_{1,2}'$,
  extending every part of those partitions with $y$ or with $\{x,y\}$, respectively.
  \item If the edge $xy$ remained uncovered, we take it as a clique
  into the covering $\cal C$.
  \end{enumerate}

  It is easy to check that $\cal C$ is a clique covering in $G$. We
  prove that it is good.

  First note that after performing Step $1$, each vertex $v\in
  V(G)-\{x,y\}$ has the same valency as in $\cal C'$. Moreover, if
  two adjacent vertices, say $u$ and $x$, have $\alpha_G(u)=
  \alpha_G(x)= 1$, then $N_G[u]=N_G[x]$ must hold.
  Hence, if $u \in V(G)-\{x,y\}$ and $\alpha_G(u)=1$, then
  $u$ is incident with only one clique from $\cal C$. Thus,
  $val_{\cal C}(u)+\frac{n}{\alpha_G(u)}=1+n$.
  If $v$ is a vertex from $V(G)-\{x,y\}$ and $\alpha_G(v)\ge 2$,
  then the valency of $v$ might increase in Step $2$ or $3$, but not in
  both. Therefore, $val_{\cal C}(v) \le val_{\cal C'}(v)+1$, and clearly $\alpha_{G'}(v)\le \alpha_{G}(v)$.
  Since $\cal C'$ is assumed to be good, these  facts together imply
 $$val_{\cal C}(v)+\frac{n}{\alpha_G(v)}\leq val_{\cal C'}(v)+1 +
 \frac{n-2}{\alpha_{G'}(v)}+\frac{2}{\alpha_{G}(v)} \le n+1.$$

 Now, consider the vertex $x$. If $\alpha_G(x)= 1$, it is covered by
 only one clique (induced by its closed neighborhood), which was
 added to $\cal C$ in Step $1$ or $2$.
In this case $val_{\cal C}(x)+\frac{n}{\alpha_G(x)}=n+ 1$.
 Also if $\alpha_G(x)\ge \frac{n}{2}$, the trivial bound  $val_{\cal
 C}(x) \le \deg(x)\le n-1$ implies the desired inequality.
Hence, we may suppose $2 \le \alpha_G(x)< \frac{n}{2}$.

Let us denote by $s$ the number of cliques covering $x$ which were
added to $\cal C$ in Step~$1$.  Choose one vertex $u_i$ with
$\alpha_{G}(u_i)=1$ from each of these $s$ cliques. The closed
neighborhoods $N[u_i]$ are pairwise different cliques. Thus, the
obtained vertex set $S$ is independent. By the definitions of $N_1'$ and
$N_{1,2}'$, there exist no edges  between $S$ and $N_1' \cup
N_{1,2}'$. Thus,
  $\alpha(G(N_1')) \le \alpha_G(x)-s$ and
  $\alpha(G( N_{1,2}')) \le \alpha_G(x)-s$.
  Also, $|N_1'|+ |N_{1,2}'| \le |N_1|+ |N_{1,2}|-s=\deg(x)-1-s$
  follows.

\begin{itemize}
 \item If $N_{1,2}\neq\emptyset$ and
$\alpha_G(y)>1$, then
 \begin{align*}
 val_{\cal C}(x)&\leq\frac{|N_1'|-\alpha(G(N_1'))}{2}+\alpha(G(N_1'))+\frac{|N_{1,2}'|-\alpha(G(N_{1,2}'))}{2}+\alpha(G(N_{1,2}'))+s\\
 &=\frac{|N_1'|+ |N_{1,2}'|}{2}+ \frac{\alpha(G(N_1'))+
 \alpha(G(N_{1,2}'))}{2}+s\\
 &\le \frac{\deg(x)-1-s}{2}+\frac{2\alpha_G(x)-2s}{2}+s \leq \frac{n-2}{2}+\alpha_G(x).
 \end{align*}
 On the other hand, our assumption $2\le \alpha_G(x)<\frac{n}{2}$
 implies that
 $\alpha_G(x)+\frac{n}{\alpha_G(x)}\leq 2+\frac{n}{2}$.
 Thus,
 $$val_{\cal C}(x)+\frac{n}{\alpha_G(x)}\leq\frac{n-2}{2}+\alpha_G(x)+\frac{n}{\alpha_G(x)}\leq
  \frac{n-2}{2}+2+\frac{n}{2}=n+1.$$

\item
If $N_{1,2}\neq\emptyset$ and $\alpha_G(y)=1$, all edges between
$N_{1,2}$ and $x$ are covered by the clique $N_G[y]$, which was
added to $\cal C$ in Step $2$ (or maybe earlier, in Step $1$).
Hence, $N_{1,2}'=\emptyset$ and we have
 \begin{align*}
 val_{\cal C}(x)&\leq\frac{|N_1'|-\alpha(G(N_1'))}{2}+\alpha(G(N_1'))+1+s\\
 &=\frac{|N_1'|}{2}+ \frac{\alpha(G(N_1')) }{2}+1+s\\
 & \le \frac{\deg(x)-1-s}{2}+\frac{\alpha_G(x)-s}{2}+1+s \leq
  \frac{n-2}{2}+\alpha_G(x).
 \end{align*}
Again, we may conclude $val_{\cal C}(x)+\frac{n}{\alpha_G(x)}\leq
n+1$.
\item
If $N_{1,2}=\emptyset$, the clique $xy$ was added to $\cal C$ in
Step $4$, and  the same estimation holds as in the previous case.
\end{itemize}

One can show similarly that $val_{\cal
C}(y)+\frac{n}{\alpha_G(y)}\leq n+1$. This completes the proof.
\end{proof}

Since for every  $v \in V(G)$, $\alpha_G(v)\le \alpha_L(G)\le
\alpha(G)$, we have the following immediate consequences.

\begin{cor}\label{cor:alpha}
Let $G$ be a graph of order $n$. Then
\begin{itemize}
\item[$(i)$]
$\lcc(G)+\frac{n}{\alpha_L(G)}\leq n+1$;
\item[$(ii)$]
$\lcc(G)+\frac{n}{\alpha(G)}\leq n+1$.
\end{itemize}
\end{cor}

On the other hand,  $val_{\cal C}(v)\geq \alpha_G(v)$, for every
arbitrary clique covering $\cal C$. Hence, $\lcc(G)\geq
\alpha_L(G)$. (But $\lcc (G) < \alpha(G)$ may be true.) Also,
it is easy to see that $\lcc(G)\geq\frac{\Delta(G)}{\omega-1}$.
 Next we observe that replacing $\lcc(G)$ with
$\alpha(G)$ or $\frac{\Delta(G)}{\omega-1}$ in Conjecture
\ref{conj2}, valid inequalities are obtained.
\begin{prop}
	If $G$ is a graph on $n$ vertices, then
	\begin{enumerate}
\item $\frac{\Delta(G)}{\omega-1}+\chi(G)\le n+1$, and equality holds if and only if
$G$ is the complete graph $K_n$ or the star $K_{1,n-1}$;

\item $\alpha(G)+\chi(G)\leq n+1$, and equality holds if and only if there
exists a vertex $v \in V(G)$ such that $N(v)$ induces a complete graph and
$V(G)\setminus N(v)$ is an independent set.
\end{enumerate}
\end{prop}
\begin{proof}
To prove (1), first note that it is showed in \cite{Finck} that
there are only two types of graphs $G$ for which $\chi(G)+\chi(\bar G)=n+1$,
\begin{description}
	\item[(a)] if $V(G)=K\cup S$ where $K$ is a clique and $S$ is an independent set,
	sharing a vertex $K \cap S = \{u\}$, or
	\item[(b)] $G$ is obtained from (a) by substituting $C_5$ into $u$.
\end{description} 
Now, we estimate $\frac{\Delta(G)}{\omega-1}+\chi(G)$ as follows. We write $\theta$
for the clique covering number (minimum number of complete subgraphs whose
union is the entire vertex set, that is the chromatic number of the
complemetary graph). Let $x$ be a vertex of degree $\Delta=\Delta(G)$. We have
$$\frac{\Delta}{\omega-1} \le \theta(G[N(x)]) \le \theta(G) \le n+1-\chi(G),$$
where the last inequality is the Nordhaus-Gaddum theorem (Theorem \ref{thm:NG}). Thus, in order to have
$\frac{\Delta}{\omega-1}+\chi=n+1$, it is necessary that $G$ is of type (a)
or (b). We shall see that (b) is not good enough, and (a) yields $G=K_n$ or
$G=K_{1,n-1}$.

Note that equality does not hold for $G=C_5$, therefore in (b) we have $k>0$.
Let $|K-u|=k$ and $|S-u|=s$ in (a). Then after substitution of $C_5$, we have
$n=k+s+5$, $\Delta \le n-1$, $\omega=k+2$ (with $k>0$), and $\chi=k+3$.
Therefore, the most favorable case is $s=0$, because increasing $s$ by 1 makes $n+1$
increase by 1, while the left-hand side of the inequality increases by at
most $1/2$. Hence, in the best case we have $n=k+5 \ge 6$, and
$$\frac{\Delta}{\omega-1} + \chi = \frac{n-1}{n-3} + n-2 < n+1$$
Now, we consider case (a). Here, again we have $k>0$ and $\Delta \le n-1$, moreover
now $n=k+s+1$, $\omega=k+1$, and $\chi=k+1$. Thus
$$\frac{\Delta}{\omega-1} + \chi \le \frac{(k+s)}{k} + k+1 \le k+s+2$$
with equality if and only if $s/k = s$, that is $k=1$ or $s=0$, where for the
case $k=1$ we also have to ensure $\Delta=s+1$. This completes the proof of (1).

To see (2), consider an independent set $A$ of cardinality $\alpha=\alpha(G)$.
A proper $(n-\alpha +1)$-coloring always exists
as we can assign color 1 to all vertices from $A$ and the further
$n-\alpha$ vertices are assigned with pairwise different colors.
Hence, $\chi(G)\le n-\alpha +1$ holds for every graph. Moreover, if the graph
induced by $V(G)\setminus A$ is not complete, we can color it properly by
using fewer than $n-\alpha$ colors that yields a proper coloring of $G$ with
fewer than $n-\alpha+1$ colors.
Therefore, $\chi(G)= n-\alpha +1$ may hold only if $V(G)\setminus A$ induces a
complete graph. In this case, $G$ is a split graph. Since split graphs are chordal and chordal graphs are perfect
\cite{dirac},
 $\omega(G)=\chi(G)=n-\alpha+1$.
Consequently, if (2) holds with equality, there exists a vertex $v\in A$
which is adjacent to  all vertices from $V(G)\setminus A$.
This vertex fulfills our conditions as $N(v)$ is a clique and
$V(G)\setminus N(v)$ is an independent set.

On the other hand, if a vertex $v'$ with such a property exists in $G$,
then the graph cannot be colored with fewer than $|N(v')|+1$ colors.
This implies $\chi= n-\alpha+1$ and completes the proof of the second statement.
\end{proof}

\section{Claw-free graphs}\label{sec:bound}

 Several related problems (say, perfect graph conjecture, to mention just the most famous one) are easier  for {\it claw-free graphs}, i.e.\ for graphs not containing $K_{1,3}$ as an induced subgraph, other problems (say, complexity of finding chromatic number) are not. (For a survey of results on claw-free graphs see e.g.\ \cite{Faudree}.) Concerning local clique cover number, R. Javadi et al.\ showed in \cite{JavadiLCC} that if $G$ is a claw-free graph then $\lcc(G)\leq c\frac{\Delta(G)}{\log(\Delta(G))}$, for a constant $c$. In this section, we are going to prove that
 Conjecture \ref{conj2} does hold for claw-free graphs.

To prove the main result of this section, we use the following
definition and theorem of Balogh et al.\ \cite{balogh}.
\begin{definition}{\cite{balogh}}
A graph $G$ is $(s, t)$-splittable if $V(G)$ can be partitioned into two sets $S$ and $T$
such that $\chi(G[S])\geq s$ and $\chi(G[T])\geq t$. For $2\leq s\leq\chi(G)-1$, we say that $G$ is $s$-splittable if $G$ is $(s,\chi(G)-s+1)$-splittable.
\end{definition}
\begin{theorem}{\em\cite{balogh}}\label{thm:alpha=2}
Let $s\geq2$ be an integer. Let $G$ be a graph with $\alpha(G)=2$ and $\chi(G)>\max\{\omega, s\}$. Then $G$ is $s$-splittable.
\end{theorem}

Now we prove:

\begin{theorem}\label{thm:claw-free}
Let $G$ be a claw-free graph with $n$ vertices. Then $\lcc(G)+\chi(G)\leq n+1$. Moreover, for every $n\ge 4$, there exist several claw-free graphs with $n$ vertices such that equality holds.
\end{theorem}
\begin{proof} We prove the theorem by induction on $n$.
For small values of $n$, it is easy to check that a claw-free graph with $n$ vertices satisfies the inequality. Also, the assertion is obvious for $\alpha(G)=1$.

Let $G$ be a claw-free graph  on $n$ vertices. First, we consider
the  case that $\alpha(G)\geq 3$. Let $T$ be an independent set of
size three. By the induction hypothesis, $G-T$ has a clique covering
 $\cal C'$ such that every vertex $x \in V(G-T)$ has
\begin{equation}\label{eq:lcc(G-T)}
 val_{\cal C'}(x) \leq (n-3)+1-\chi(G-T)\leq n-2-(\chi(G)-1)=n-1-\chi(G).
\end{equation}

Now, for every vertex $u\in T$, partition $N(u)$ into the $\chi(\overline{G[N(u)]})$ vertex-disjoint cliques. Then, add vertex $u$ to each clique to cover all the edges incident to $u$. These cliques along with cliques in an optimum clique covering of $G-T$ form a clique covering, say $\cal C$, for $G$. Let $u\in T$ and $x\in G-T$. Then we have
\begin{align*}
val_{\cal C}(u)&=\chi(\overline{G[N(u)]})\leq\chi(\overline{G})\leq n+1-\chi(G),\\
val_{\cal C}(x)&\leq val_{\cal C'}(x) +|N_G(x)\cap T|.
\end{align*}
Since $G$ is claw free, $|N_G(x)\cap T|\leq 2$.
 Thus, by Inequality~\eqref{eq:lcc(G-T)}, $\lcc(G)\leq n+1-\chi(G)$.

Consider now the case $\alpha(G)=2$. By Proposition \ref{prop:omega}
we may assume that $\chi(G)>\omega(G)$. Moreover, as the
statement clearly holds when $\chi(G) \le 2$, we may also suppose
that $\chi(G) \ge 3$. Then Theorem \ref{thm:alpha=2} with $s=2$
implies that $V(G)$ can be partitioned into two parts, say $A$ and
$B$, such that $\chi(G[A])\geq2$ and $\chi(G[B])\geq \chi(G)-1$. We
assume, without loss of generality, that $A=\{u_1,u_2\}$, where the
vertices $u_1$ and $u_2$
 are adjacent.
%
 Then $\chi(G-\{u_1,u_2\})\geq \chi(G)-1$.

We will use the  notation $N_1:=N(u_1)-N[u_2]$,
$N_2:=N(u_2)-N[u_1]$, and  $N_{1,2}:=N(u_1)\cap N(u_2)$.
Since $G$ is claw-free,
 $N_i\cup \{u_i\}$ induces a clique for $i=1,2$.
 Starting with an optimal clique covering $\cal C''$ for $G-\{u_1,u_2\}$, we will
 construct  a clique covering  $\cal C$ for
  $G$ such that $val_{\cal C}(v) \le n+1-\chi(G)$ holds for every
  vertex $v$.

  If $N_{1,2}=\emptyset$, then  ${\cal C}:= {\cal C''} \cup \{N_1\cup \{u_1\},
N_2\cup \{u_2\}, \{u_1,u_2\} \}$ is a clique covering for $G$. We
observe that $val_{\cal C} (u_i) \le 2$ holds for $i=1,2$ and
$$val_{\cal C} (v) \le val_{\cal C''} (v)+1 \le n-1-\chi(G
-\{u_1,u_2\})+1 \le n-\chi(G)$$ for each vertex $v$ from
$V(G-\{u_1,u_2\})$. Hence, $\lcc(G)\leq n+1-\chi(G)$.

Otherwise, if  $N_{1,2}\neq \emptyset$, partition $N_{1,2}$ into at
most $\chi(\overline{G-\{u_1,u_2\}})$ cliques and extend each of them with
 the vertices $u_1$ and $u_2$. These cliques together with $N_1\cup
\{u_1\}$, $N_2\cup \{u_2\}$,
   and with the cliques in $\cal C''$ form a clique covering of $G$.
   We show that this clique covering $\cal C$ satisfies $val_{\cal C}(x)\leq n+1-\chi(G)$
    for every vertex $x\in V(G)$.
   Note that $val_{\cal C} (u_1)\leq \chi(\overline{G-\{u_1,u_2\}})+1$, thus the Nordhaus-Gaddum
   inequality for chromatic number implies
    $$val_{\cal C}(u_1)\leq(n-2)+1-\chi(G-\{u_1,u_2\})+1\leq n-\chi(G-\{u_1,u_2\})\leq n+1-\chi(G).$$
   Similarly, we have $val_{\cal C}(u_2)\leq n+1-\chi(G)$. For $v\in V(G-\{u_1,u_2\})$,
$$
val_{\cal C}(v)\leq val_{\cal C''}(v)+1
 \leq (n-2)+1-\chi(G-\{u_1,u_2\})+1
 \leq n-\chi(G)+1 .
  $$

Finally, we note that $K_n$, $K_n-K_2$, and $K_n-K_{1,2}$ are examples of claw-free graphs with $n$ vertices such that $\lcc(G)+\chi(G)=n+1$.
\end{proof}

\section{A Nordhaus-Gaddum type inequality}
A {\it clique partition} of the edges of a graph $ G $ is a family of cliques
 such that every edge of $ G $ lies in exactly one member of the family. The {\it sigma clique partition number} of $G$, $\scp(G)$, is the smallest integer $k$ for which there exists a clique partition of $E(G)$ where  the sum of  the sizes of its cliques is at most $k$.

It was conjectured by G. O. H. Katona and T. Tarj\'{a}n, and proved in the papers
 \cite{chung81,kahn81,gyori80}, that for every graph $G$ on $n$ vertices, $\scp(G)\leq \lfloor n^2/2\rfloor$ holds, with equality if and only if $ G $ is the complete bipartite graph $K_{\lfloor n/2\rfloor, \lceil n/2\rceil } $.

Also, this parameter relates to a number of other well-known problems (see \cite{SCPandPBD}). The second author and R. Javadi proved the following Nordhaus-Gaddum type theorem for $\scp$.

\begin{theorem}{\em\cite{IMC45}}
Let $G$ be a graph with $n$ vertices. Then
\begin{align*}
\frac{31}{50}n^2+O(n)\leq\max\{\scp(G)+\scp(\overline{G})\}\leq \frac{9}{10} n^2+O(n),\\
\frac{12}{125}n^4+O(n^3)<\max\{\scp(G)\cdot
\scp(\overline{G})\}<\frac{81}{400}n^4+O(n^3).
\end{align*}
\end{theorem}

In the following result we improve the upper bounds, from 0.9 to less than 0.77
 and from 0.2025 to less than 0.15.

\begin{theorem}\label{thm:scp}
For every graph $G$ with $n$ vertices,
 $$\scp(G)+\scp(\overline{G})\leq \frac{1203}{1568} n^2+o(n^2) < 0.76722 \, n^2+o(n^2)$$
  and
 $$\scp(G)\cdot\scp(\overline{G})\leq \frac{1447209}{9834496} n^4+o(n^4)
  < 0.1471564 \, n^4+o(n^4) \, .$$
\end{theorem}
\begin{proof}
Substantially improving on earlier estimates, P. Keevash and B. Sudakov \cite{KeeSud}
 proved via a computer-aided calculation
  that every edge 2-coloring of $K_n$ contains at least $cn^2-o(n^2)$
 mutually edge-disjoint monochromatic triangles,\footnote{In the Abstract
  of \cite{KeeSud} the authors announce the lower bound $n^2\!/13$, and in their
   Theorem 1.1 they state $n^2\!/12.89$ (the rounded form of $\frac{9}{116}n^2$,
   but actually on p.~212 they prove the even better lower bound displayed above.)}
  where
 $$
   c = \frac{13}{196} + \frac{1}{84} - \frac{1}{1568} = \frac{365}{4704} \, .
 $$
In our context this means that we can select approximately $cn^2$ triangles
 which together cover $3cn^2$ edges in $G$ and $\overline{G}$ at the cost of
 $3cn^2$.
The remaining edges will be viewed as copies of $K_2$ in the clique partition
 to be constructed; they are counted with weight 2 in scp. In this way we obtain
 $$
   \scp(G)+\scp(\overline{G})\leq \left( 1-3c \right) n^2 + o(n^2)
     = \frac{1203}{1568} n^2 + o(n^2) \, .
 $$
This also implies the upper bound on $\scp(G)\cdot\scp(\overline{G})$.
\end{proof}

\begin{remark}
The smallest number of cliques in a clique partition of $ G $ is called the {\it clique partition number} of $ G $. As a Nordhaus-Gaddum type inequality for parameter $\cp$,
D.~de~Caen et al.\ proved in \cite{caen} that
\begin{align*}
\cp(G)+\cp(\overline{G})&\leq\frac{13}{30}n^2-O(n) \approx 0.43333\, n^2-O(n) \, , \\
\cp(G)\cdot\cp(\overline{G})&\leq\frac{169}{3600}n^4+O(n^3) \approx 0.0469444\, n^2+O(n^3) \, .
\end{align*}
Note that if it is possible to select some $k$ edge-disjoint complete subgraphs
 in $G$ and $\overline{G}$ which together cover $m$ edges, then
 $\cp(G)+\cp(\overline{G})\le {n\choose 2}+k-m$.
As observed within the proof of Theorem \ref{thm:scp}, the choices
 $k=\frac{365}{4704}n^2-o(n^2)$ and $m=3k$ are feasible for every $G$ on $n$ vertices, thus
\begin{align*}
\cp(G)+\cp(\overline{G})&\leq\left(\frac12-\frac{365}{2352}\right)n^2+o(n^2)
  =\frac{811}{2352}n^2+o(n^2) < 0.344813\, n^2+o(n^2) \,,\\
\cp(G)\cdot\cp(\overline{G})&\leq\frac{657721}{22127616}n^4+o(n^4) < 0.029724\, n^4+o(n^4) \,.
\end{align*}
 These upper bounds improve the results of \cite{caen}.
\end{remark}

\end{document}